\documentclass[12pt, reqno]{amsart}
\usepackage[dvipsnames]{xcolor}
\usepackage{mathtools}
\mathtoolsset{showonlyrefs}
\usepackage{etoolbox}

\usepackage{lipsum} 
\usepackage{tikz,xcolor}

\definecolor{lime}{HTML}{A6CE39}
\DeclareRobustCommand{\orcidicon}{
	\begin{tikzpicture}
	\draw[lime, fill=lime] (0,0) 
	circle [radius=0.16] 
	node[white] {{\fontfamily{qag}\selectfont \tiny ID}};
	\draw[white, fill=white] (-0.0625,0.095) 
	circle [radius=0.007];
	\end{tikzpicture}
	\hspace{-2mm}
}
\foreach \x in {A, B}{\expandafter\xdef\csname orcid\x\endcsname{\noexpand\href{https://orcid.org/\csname orcidauthor\x\endcsname}
			{\noexpand\orcidicon}}
}
\usepackage{blindtext}
\usepackage{latexsym}
\usepackage{txfonts}
\usepackage{lmodern}
\usepackage[pagewise]{lineno}
\usepackage[utf8]{inputenc}
\usepackage{exscale}
\usepackage{amsmath}
\usepackage{amssymb}
\usepackage{amsfonts}
\usepackage{mathrsfs}
\usepackage{amsbsy}
\usepackage{relsize}
\usepackage{comment}
\PassOptionsToPackage{reqno}{amsmath}
\usepackage{esint} 
\usepackage{amssymb} 
\usepackage{stmaryrd}
\usepackage{cite} 

\usepackage{tikz}

\usepackage[a4paper,left=30mm,right=30mm,top=30mm,bottom=30mm,marginpar=20mm]{geometry}
\everymath{\displaystyle}
\usepackage{mathtools}
\DeclareMathOperator*{\esssup}{ess\,sup}

\usepackage{soul}
\usepackage{color}

\usepackage{slashed}
\PassOptionsToPackage{hidelinks}{hyperref}

\definecolor{mypink1}{rgb}{0.858, 0.188, 0.478}
\definecolor{mypink2}{RGB}{219, 48, 122}
\definecolor{mypink3}{cmyk}{0, 0.7808, 0.4429, 0.1412}
\definecolor{mygray}{gray}{0.6}

\definecolor{gris75}{gray}{0.25}
\definecolor{violet}{rgb}{0.5,0,0.5}

\definecolor{BrickRed}{rgb}{0.58, 0.0, 0.83}

\definecolor{armygreen}{rgb}{0.29, 0.33, 0.13}

\definecolor{brass}{rgb}{0.71, 0.65, 0.26}
\definecolor{antiquefuchsia}{rgb}{0.57, 0.36, 0.51}
\definecolor{amethyst}{rgb}{0.6, 0.4, 0.8}

\definecolor{mauvetaupe}{rgb}{0.57, 0.37, 0.43}
\definecolor{mypink1}{rgb}{0.858, 0.188, 0.478}
\definecolor{mypink2}{RGB}{219, 48, 122}
\definecolor{mypink3}{cmyk}{0, 0.7808, 0.4429, 0.1412}
\definecolor{mygray}{gray}{0.6}
\definecolor{venetianred}{rgb}{0.78, 0.03, 0.08}
\definecolor{sapphire}{rgb}{0.03, 0.15, 0.4}
\definecolor{utahcrimson}{rgb}{0.83, 0.0, 0.25}
\definecolor{trueblue}{rgb}{0.0, 0.45, 0.81}
\definecolor{carminered}{rgb}{1.0, 0.0, 0.22}
\definecolor{cobalt}{rgb}{0.0, 0.28, 0.67}
\definecolor{cornflowerblue}{rgb}{0.39, 0.58, 0.93}
\usepackage[colorlinks, linkcolor=cornflowerblue, citecolor=blue, urlcolor=blue, hypertexnames=false]{hyperref}

\newtheorem{theorem}{Theorem}[section]
\newtheorem{lemma}[theorem]{Lemma}

\newtheorem{proposition}[theorem]{Proposition}
\newtheorem{remark}[theorem]{Remark}

\numberwithin{equation}{section}

\numberwithin{bbb}{section}

\def\R{\mathbb{R}}

\title[Inhomogeneous parabolic equation]{An inhomogeneous nonlinear parabolic equation with concave nonlinearity}
\author[Miguel Loayza]{Miguel Loayza\orcidA{}}
\address{\small Departamento de Matem\'atica, Universidade Federal de Pernambuco - UFPE, Recife, PE, Brazil}
\email{\textcolor{blue}{\it miguel@dmat.ufpe.br}}
\author[Mohamed Majdoub]{Mohamed Majdoub\orcidB{}}
\address{\small Department of Mathematics, College of Science, IAU, P.O. Box 1982, Dammam, Saudi Arabia.\newline Basic and Applied Scientific Research Center, IAU, P.O. Box 1982, 31441, Dammam, Saudi Arabia.}
\email{\textcolor{blue}{\it mmajdoub@iau.edu.sa}}
\email{\textcolor{blue}{\it med.majdoub@gmail.com}}

\begin{document}

\begin{abstract}
We establish both the existence and uniqueness of non-negative global solutions for the nonlinear heat equation
$u_t-\Delta u=|x|^{-\gamma}\,u^q$, $0<q<1$, $\gamma>0$ in the whole space $\mathbb{R}^N$,  and for {non-negative} initial data $u_0\in C_0(\mathbb{R}^N)$.
\end{abstract}
\thanks{\em Miguel Loayza was partially supported by CAPES-PRINT, 88881.311964/2018-01, MATHAMSUD, 88881.520205/2020-01, CNPq - 313382/2023-9.}
\maketitle

\section{Introduction}
Consider the singular nonlinear parabolic problem
\begin{equation}\label{In.uno}
\left \{ \begin{array}{rl}
u_t-\Delta u&= |x|^{-\gamma}\, u^q\quad  \mbox{ in }\quad  \mathbb{R}^N \times (0,{\infty}),\\
u({x,0})&=u_0{(x)}\geq 0 \quad\mbox{ in }\quad \mathbb{R}^N,
\end{array}\right.
\end{equation}
where $0<q<1$, $\gamma>0$, and $u_0\in C_0(\mathbb{R}^N)$, where $C_0(\mathbb{R}^N)$ represents the closure in $L^\infty(\mathbb{R}^N)$ of $\mathcal{D}(\mathbb{R}^N)$, the set of infinitely differentiable functions with compact support in $\mathbb{R}^N$. Our main interest in this paper is to analyze the existence and uniqueness of nonnegative global solutions of \eqref{In.uno} in the class $L^\infty (\mathbb{R}^N \times (0,T))$.

{The initial value problem \eqref{In.uno} has attracted considerable attention in the mathematical community. The case $\gamma = 0$ was studied in \cite{AE87}. When $|x|^{-\gamma}$ is replaced by a function $\mathbf{a}(x)$ on bounded domains, the problem was addressed in the pioneering work \cite{Watanabe} for continuous $\mathbf{a}$, and in \cite{Loa2006} for $\mathbf{a}$ belonging to Lebesgue spaces.}

{It is worth noting that extensive research has been conducted on the nonlinear heat equation with convex nonlinearities (i.e., $q > 1$). The monographs \cite{Hu, QS} provide a comprehensive overview of the most well-established results in this area. For further reading, see \cite{BC, CIT, CITT, SW, Tayachi, Weissler} and the references therein.}

{The existence and uniqueness of solutions for problem \eqref{In.uno}, with $\gamma > 0$, a more general nonlinearity than $u^q$, and nonnegative initial data in $L^r(\mathbb{R}^N)$ and $L_{ul}^r(\mathbb{R}^N)$, have been studied in \cite{R2, R3}. Here, $L_{ul}^r(\mathbb{R}^N)$ denotes the uniformly local Lebesgue space. However, the uniqueness results obtained in these works cannot be directly applied to problem \eqref{In.uno}. The primary difficulty in establishing uniqueness for the solutions of \eqref{In.uno} lies in the fact that the nonlinearity $u^q$ is not a Lipschitz function. This issue was resolved for the case $\gamma = 0$ in \cite{AE87}, where it was shown that every positive solution $u$ of problem \eqref{In.uno} satisfies the following lower estimate:}\begin{equation}
\label{In.b}
u(t)\geq [(1-q)t]^{1/(1-q)}=\underline{u}(t).
\end{equation}
Moreover, as is easily verified, $\underline{u}$ is a global solution of problem \eqref{In.uno} with $u_0=0$ and $\gamma=0$. The second difficulty is the presence of the singular weight $\mathbf{a}(x)=|x|^{-\gamma}$ with $\gamma>0$. In this case, $\underline{u}$ does not satisfy (\ref{In.uno}), necessitating a new lower-bound estimate. The forthcoming Theorem \ref{UN} furnishes the subsequent lower-bound estimate
$$u{(x,t)}\geq [(1-q) \eta_0]^{1/(1-q)} t^{1/(1-q)} (|x|+\sqrt{t})^{-\gamma/(1-q)}=:w(x,t),$$
where $\eta_0$ is the constant given by \eqref{eta-0} below. Clearly, when $\gamma=0$, we obtain $w(x,t)=\underline{u}(t)$. Although $w$ does not constitute a solution to the the problem (\ref{In.uno}), it does serve as a subsolution with initial value $u_0=0$, as shown in Lemma \ref{Lema.subs} below.

As is a common practice, we analyze \eqref{In.uno} through the associated integral equation
\begin{equation}
\label{integral}
u(t)= {\mathbf S}(t)u_{0}+\displaystyle\int_{0}^{t}\,{\mathbf S}_{\gamma}(t-\sigma)\,u^q(\sigma)\,d\sigma,
\end{equation}
where $\{{\mathbf S}(t)\}_{t\geq 0}$ denotes the heat semigroup and ${\mathbf S}_{\gamma}(t)={\mathbf S}(t)|\cdot|^{-\gamma}$. We say that $u\in C([0,\infty); C_0 (\mathbb{R}^N))$ is a non-negative solution to problem (\ref{In.uno}) if the integral equation (\ref{integral}) is verified for every $t\in [0,T]$ with $0<T<\infty$. This solution is said to be  positive if additionally  $u>0$ in $\mathbb{R}^N \times (0,\infty)$. 
Henceforth, we adopt the notation 
\begin{equation}\label{G.N}
\gamma_N=\min\{2,N\},    
\end{equation}
which will be employed frequently. We now proceed to present our first result as follows.
\begin{theorem}
\label{GE}
Let $u_0 \in C_0(\mathbb{R}^N), u_0\geq 0$,   $0<q<1$ and $0<\gamma<\gamma_N$. Then, there exists a non-negative solution for problem (\ref{In.uno}).
\end{theorem}
\begin{remark}
{\rm Here are some comments about Theorem \ref{GE}.
\begin{enumerate}
    \item[a)] The same result also applies to the case of $q=1$, which corresponds to the linear heat equation with a singular potential.
    \item[b)] For $\gamma=-2$ the existence of solutions for problem (\ref{In.uno}) fails in general, see \cite{R1}, \cite[Proposition 6.6]{R4}  and \cite{HT2001} (for bounded domains).
    \item[c)] {From \cite[Theorem 1.1(ii)]{R2}, it follows that problem \eqref{In.uno} admits a solution in $L^\infty((0,T), L^r_{ul}(\mathbb{R}^N))$ provided that $u_0 \in L^r_{ul}(\mathbb{R}^N)$, $u_0 \geq 0$, $r > 1$, $0 < q < 1$, and $0 < \gamma < \gamma_N$. It is worth noting that, while an approximation method is employed to prove Theorem \ref{GE}, a supersolution method is used in \cite{R2}.}
\end{enumerate}}
 
\end{remark}

In the next result, we obtain an estimate from below for non-trivial solutions of problem (\ref{In.uno}) and we use it to establish a uniqueness result. 

\begin{theorem} 
\label{UN}
Assume that $0<q<1$, $0<\gamma<\gamma_N$, $u_0 \in C_0(\mathbb{R}^N)$, $0\neq u_0\geq 0$. Let $u \in C([0,\infty); C_0(\mathbb{R}^N))$ {be} a non-negative solution of problem (\ref{In.uno}). Then
\begin{enumerate}
\item[(i)]  $u$ is positive and verifies the following estimate
\begin{equation}\label{Est.pr}
u(x,t)\geq [(1-q) \eta_0]^{1/(1-q)} t^{1/(1-q)} (|x|+\sqrt{t})^{-\gamma/(1-q)}={{:}}w(x,t),
\end{equation}
 for all $x\in \mathbb{R}^N$ and $t\geq 0$, where
{
 \begin{equation}
    \label{eta-0}
   {\eta_0=\eta_0(\gamma)=(4\pi)^{-N/2}\displaystyle\int_{\mathbb{R}^N} \exp \left( -\frac{|z|^2}{4}\right)(1+|z|)^{-\frac{\gamma}{1-q}}dz.}
 \end{equation}
}
Moreover, $w\in C([0,\infty); C_0(\mathbb{R}^N))$ verifies (\ref{integral}), with $\leq $ {instead} of $=$, and $u_0=0$. In other words, $w$ is a sub-solution of problem (\ref{In.uno}) with $u_0=0$.

\item[(ii)]  There exists $\gamma^* \in (0, \gamma_N)$ so that for $0<\gamma<\gamma^*$, $u$ is the unique solution of the problem (\ref{In.uno}).
\end{enumerate}

\end{theorem}
\begin{remark}  
{\rm {Below are a few remarks regarding Theorem \ref{UN}.}
\begin{enumerate}
\item[a)] {When $\gamma=0$, we find that $\eta_0$ equals $1$ and the expression for $w{(x,t)}$ simplifies to $[(1-q)t]^{1/(1-q)}$. In other words, $w$ corresponds to the lower bound provided in \eqref{In.b}.}
\item[b)] {When $\gamma=0$ and $u_0\neq 0$, it has been established in \cite{AE87} that problem \eqref{In.uno} possesses a unique solution. Hence, we generalize this result to small values of $\gamma > 0$.}
\end{enumerate}
}
\end{remark}
We conclude the introduction with an outline of the paper. In Section \ref{S2}, we recall some useful tools and introduce some auxiliary results. Section \ref{S3} is devoted to the proof of the existence result stated in Theorem \ref{GE}. {In Section \ref{S4}, we present the proof of our uniqueness result, which is stated as Theorem \ref{UN}. Additionally, we denote the norm in the Lebesgue space $L^p(\mathbb{R}^N)$ as $\|\cdot\|_p$, where $1 \leq p \leq \infty$.}
\section{Preliminaries}
\label{S2}

Let $\{{\mathbf S}(t)\}_{t\geq 0}$ be the linear heat semigroup defined by ${\mathbf S}(t)\,\varphi=G_t\star\varphi, t>0,$ where $G_t$ is the heat kernel given by
$$
G_t(x)=\left(4\pi t\right)^{-N/2}\,{\rm e}^{-|x|^2/(4t)},\;\;\;t>0,\;\;\;x\in\R^N.
$$
In the subsequent lemma, we gather some basic facts regarding the heat semigroup $\mathbf{S}(t)$ that will be helpful for our purposes.
\begin{lemma} \label{S-t} Let $a>0, t_0>0,$ and $0\leq v_0\in C_0(\R^N)$ with $v_0\neq 0$. Then, the following assertions hold true:
    \begin{equation}
       \label{smooth-effect}
    \|{\mathbf S}(t)\,\varphi\|_{p}\leq \|\varphi\|_p,\quad \mbox{for}\;\;\; t>0,\;\;\; 1\leq p\leq \infty,\;\;\; \varphi\in L^p (\mathbb{R}^N).
    \end{equation}
    \begin{equation}
    \label{Identity}
{\mathbf S}(t)\left(e^{-a|\cdot|^2}\right)(x)=(2\pi)^N\,(1+4at)^{-N/2}\,e^{-\frac{a|x|^2}{1+4at}},\quad t>0,\;\; x\in\R^N.
\end{equation}
\begin{equation}
\label{S-t-0}
\left[{\mathbf S}(t_0)v_0\right](x)\geq  C' \,e^{-\frac{|x|^2}{2t_0}}\quad\mbox{for all}\quad x\in\R^N,
    \end{equation}
    where $C'= (4\pi\,t_0)^{-N/2}\displaystyle\int_{\R^N}e^{-\frac{|y|^2}{2t_0}}v_0(y) dy>0$.
\end{lemma}
\begin{proof}[{Proof of Lemma \ref{S-t}}]
    {The contractivity of the heat semigroup in $L^p$ spaces, as stated in \eqref{smooth-effect}, is a consequence of Young's inequality, coupled with the property $\|G_t\|_1=1$ for all $t>0.$ 
    To derive \eqref{Identity}, we note that
    $$
    \mathscr{F}\left({\mathbf S}(t)\left(e^{-a|\cdot|^2}\right)\right)(\xi)=\left(\frac{\pi}{a}\right)^{N/2}\,e^{-(t+\frac{1}{4a})|\xi|^2},
    $$
    where $\mathscr{F}$ denotes the Fourier transform.\footnote{{Here, we adopt the definition of the Fourier transform as  $\widehat{f}(\xi)=\mathscr{F}f(\xi):=\displaystyle\int_{\R^N}e^{-ix\cdot \xi}f(x)dx.$}} Applying the inverse Fourier transform, we obtain
    \begin{eqnarray*}
     {\mathbf S}(t)\left(e^{-a|\cdot|^2}\right)(x)&=& \left(\frac{\pi}{a}\right)^{N/2}\,\left(\frac{\pi}{t+\frac{1}{4a}}\right)^{N/2}\,e^{-\frac{|x|^2}{4(t+\frac{1}{4a})}}\\
     &=&(2\pi)^N (1+4at)^{-N/2}\,e^{-\frac{a|x|^2}{1+4at}}.
    \end{eqnarray*}
Finally, in order to derive \eqref{S-t-0}, we employ the elementary inequality $\frac{|x-y|^2}{4t_0}\leq \frac{|x|^2}{2t_0}+\frac{|y|^2}{2t_0}$. Hence,
    \begin{eqnarray*}
      \left[{\mathbf S}(t_0)v_0\right](x)&=&(4\pi t_0)^{-N/2}\,\displaystyle\int_{\R^N}\,e^{-\frac{|x-y|^2}{4t_0}}\,v_0(y)\,dy\\
      &\geq& (4\pi t_0)^{-N/2}\,e^{-\frac{|x|^2}{2t_0}}\displaystyle\int_{\R^N}\,e^{-\frac{|y|^2}{2t_0}}\,v_0(y)\,dy\\
      &=&C' \,e^{-\frac{|x|^2}{2t_0}}.
    \end{eqnarray*}
    }
\end{proof}
 By following a similar argument as presented in \cite[Theorem 1, p. 47]{Evans} and utilizing the density of $\mathcal{D}(\mathbb{R}^N)$, we can deduce the continuity of ${\mathbf S}(t)$ on $C_0(\mathbb{R}^N)$.
\begin{lemma}
Consider $u_0\in C_0(\R^N).$ Then 
\begin{equation}
    \label{S-t-C-0-1}
    {\mathbf S}(\cdot )u_0\in C([0,\infty); C_0(\R^N)).
\end{equation}
\end{lemma}
\begin{proof}
  For the convenience of the reader, we provide a proof here. Let $\varepsilon>0$. Due to the fact that $C_0(\R^N)=\overline{\mathcal{D}(\mathbb{R}^N)}^{L^\infty}$, we can find $\varphi\in \mathcal{D}(\mathbb{R}^N)$ such that
  \begin{equation}
    \label{Dens-0}
    \|u_0-\varphi\|_\infty \leq \varepsilon. 
  \end{equation}
To show that ${\mathbf S}(t)u_0 \in C_0(\mathbb{R}^N)$ for every $t > 0$, it suffices to prove that there exists $M > 0$ such that $|[{\mathbf S}(t)u_0](x)| \leq \varepsilon$ whenever $|x| \geq M$. To establish this, assume that the support of $\varphi$ is contained in $B_R(0)$, the ball of radius $R > 0$ centered at the origin. 

For $|x| \geq 2R$ and $|y| \leq R$, we have $|x - y| \geq |x| - |y| \geq R$. Using this, we can estimate:
$$
\begin{aligned}
|[{\mathbf S}(t)\varphi](x)| &\leq (4\pi t)^{-N/2} \int_{\mathbb{R}^N} \exp\left(-\frac{|x - y|^2}{4t}\right) |\varphi(y)| \, dy \\
&\leq (4\pi t)^{-N/2} \exp\left(-\frac{R^2}{4t}\right) \int_{|y| \leq R} |\varphi(y)| \, dy \\
&\leq (4\pi t)^{-N/2} \exp\left(-\frac{R^2}{4t}\right) \|\varphi\|_{1}.
\end{aligned}
$$
For $R$ sufficiently large, the term $(4\pi t)^{-N/2} \exp\left(-\frac{R^2}{4t}\right) \|\varphi\|_{1}$ can be made smaller than $\varepsilon$. Consequently, applying the contraction property of ${\mathbf S}(t)$ (as given in \eqref{smooth-effect}), we obtain
$$
\begin{aligned}
|[{\mathbf S}(t)u_0](x)| &\leq |[{\mathbf S}(t)(u_0 - \varphi)](x)| + |[{\mathbf S}(t)\varphi](x)| \\
&\leq \|u_0 - \varphi\|_{\infty} + \varepsilon\\
&\leq 2\varepsilon.
\end{aligned}
$$
The continuity of ${\mathbf S}(t)$ for $t > 0$ is a direct consequence of its continuity at $t = 0$. Thus, by combining the density property \eqref{Dens-0} with the contraction property \eqref{smooth-effect}, we obtain
\begin{equation}
      \label{Dens-1}
      \begin{split}
          \|{\mathbf S}(t)u_0-u_0\|_\infty &\leq \|{\mathbf S}(t)\varphi-\varphi\|_\infty+\|{\mathbf S}(t)(u_0-\varphi)\|_\infty+\|u_0-\varphi\|_\infty\\
          &\leq 2\varepsilon+\|{\mathbf S}(t)\varphi-\varphi\|_\infty.
      \end{split}
  \end{equation}
  Since $\varphi\in \mathcal{D}(\R^N)$, we have ${\mathbf S}(t)\varphi{\rightarrow}\; \varphi$ in $L^\infty(\R^N)$ as $t\to 0^+$. Therefore, it follows from \eqref{Dens-1} that
$$
      \limsup_{t\to 0}\,\|{\mathbf S}(t)u_0-u_0\|_\infty\leq 2\varepsilon.
$$
This completes the proof of \eqref{S-t-C-0-1} since $\varepsilon>0$ is arbitrary.
\end{proof}
\begin{remark}
 {\rm 
 {Although the heat semigroup ${\mathbf S}(t): L^\infty(\mathbb{R}^N) \to L^\infty(\mathbb{R}^N)$ is a contraction mapping as stated in \eqref{smooth-effect}, it is not continuous at $t=0$. 
 To illustrate this, consider the case $ N = 1 $ and define the function  
$$
u_0(x) = 
\begin{cases} 
1 & \text{if } x > 0, \\
0 & \text{if } x \leq 0.
\end{cases}
$$  
It is straightforward to verify that for all $ t > 0 $,  
$$
\|{\mathbf S}(t)u_0 - u_0\|_\infty = \frac{1}{2}.
$$  
This shows that the operator $\mathbf{S}(t)$  fails to be continuous at $t = 0$ in the $L^\infty$ norm.}}
\end{remark}
For $\gamma \geq 0$, let ${\mathbf S}_{\gamma}$ be defined as
$$
{\left[{\mathbf S}_{\gamma}(t)\varphi\right](x)=(4\pi\,t)^{-N/2}\displaystyle\int_{\R^N}\,\exp{\left(-\frac{|x-y|^2}{4t}\right)}|y|^{-\gamma}\,\varphi(y)\, dy,\;\;\;t>0,\;\; x\in\R^N.}
$$
To handle the nonlinear term in \eqref{In.uno}, we rely on a crucial estimate established in \cite{SW}.
\begin{proposition}
\label{Key}
Let $0< \gamma<N$ and $1<q_1, q_2\leq \infty$ satisfy the following
$$
\frac{1}{q_2}<\frac{\gamma}{N}+\frac{1}{q_1}<1.
$$
For any $t>0$, the following operators are bounded
$$
\begin{array}{ll}
\mathbf{S}_\gamma (t): L^{q_1}(\mathbb{R}^N) \to L^{q_2}(\mathbb{R}^N)& \mbox{ if } 1<q_2<\infty,\\
\mathbf{S}_\gamma (t): L^{q_1}(\mathbb{R}^N) \to C_0(\mathbb{R}^N)& \mbox{ if } q_2=\infty.
\end{array}
$$
Furthermore, 
\begin{equation}
\label{SS}
\|{\mathbf S}_{\gamma}(t)\varphi\|_{q_2}\leq C\,t^{-\frac{N}{2}\left(\frac{1}{q_1}-\frac{1}{q_2}\right)-\frac{\gamma}{2}}\,\|\varphi\|_{q_1},\;{\text{for}\;\;\; t>0,\;\; \varphi\in L^{q_1} (\mathbb{R}^N)}
\end{equation}
where $C$ is a constant depending only on $N, \gamma, q_1$ and $q_2$.
\end{proposition}

We shall use the following technical result.
\begin{lemma}
\label{maxim}
{Consider a $C^1$-function $f: [0,\infty)\to\R$ with the property that $f'(r) \leq 0$ for all $r>0$. For any $t>0$, define
$$F(x)=(4\pi t)^{-N/2} \displaystyle\int_{\mathbb{R}^N} \exp \left (-\frac{|x-y|^2}{4t} \right) f(|y|)dy,\;\;\; x\in \mathbb{R}^N.$$}
Then $F$ reaches its maximum at $x=0$.
\label{Pinsky}
\end{lemma}
{The proof of Lemma \ref{maxim} follows a similar approach to that presented in \cite[Lemma 6]{Pinsky}. Therefore, we omit the detailed steps of the proof.}

We also make use of the following singular Gronwall inequality.
\begin{lemma}
\label{SGI}
Let $T>0$ and $0\leq \psi\in L^\infty([0,T])$ satisfy
$$
\psi(t)\leq A+M\displaystyle\int_0^t (t-\sigma)^{-\alpha}\psi(\sigma)\,d\tau, \quad \quad \mbox{for a.e.}\quad 0\leq t\leq T,
$$
where $0\leq \alpha<1$, {$A\geq 0$, and $M>0$}. Then, there exists an explicit function $C(M, \alpha, t)$ such that
\begin{equation}
\label{SGI2}
\psi(t)\leq A\, C(M, \alpha, t),\quad \mbox{for a.e.}\quad 0\leq t\leq T.
\end{equation}
In particular, if $A=0$ then $\psi\equiv 0$.
\end{lemma}
\begin{remark}
~
{\rm 
    \begin{itemize}
    \item[a)] Lemma \ref{SGI} is commonly recognized as the {\it Gronwall-Henry lemma} or the {\it singular Gronwall lemma}. Detailed discussions and proofs can be found in various sources such as \cite[p.140]{Henry}, \cite[Theorem 3.3.1]{Amann}, and \cite[Lemma 8.1.1]{CazHar}. For further insights and applications, interested readers can also refer to \cite[Theorem 3.2, p. 696]{Webb2019}.
    \item[b)] {An explicit estimate for the function $C(M,\alpha,t)$ in \eqref{SGI2} is well-established. Specifically, from \cite[Theorem 3.3.1]{Amann}, for any $\epsilon > 0$, there exists a constant $c = c(\alpha, \epsilon)$ such that  
$$
C(M,\alpha, t) = 1 + c\, Mt^{1-\alpha} \exp\left[(1+\epsilon)\mu(\alpha)t\right],  
$$
where $\mu(\alpha) = \left(M\Gamma(1-\alpha)\right)^{1/(1-\alpha)}$.}
    \item[c)] When $\psi$ is assumed to be more regular, such as $\psi\in C([0,T])$, the estimate \eqref{SGI2} can be further refined, as demonstrated in \cite[Theorem 3.1, p. 537]{DM1986}. Specifically, we have    
\begin{equation}
    \label{SGI-Refined}
    \psi(t)\leq A\, {\mathcal E}_{1-\alpha}\left(M \Gamma(1-\alpha)\,t^{1-\alpha}\right),\quad 0\leq t\leq T,
\end{equation}
where ${\mathcal E}_{1-\alpha}$ is the Mittag-Leffler function which is a generalization of the exponential function ($\alpha=0$) and it is defined by
$$
{\mathcal E}_{1-\alpha}(z)=\sum_{n=0}^{\infty}\,\frac{z^n}{\Gamma(n(1-\alpha)+1)}.
$$
 \item[d)] {As noted in \cite{DM1986}, the estimate \eqref{SGI-Refined} is sharp under the assumption that $\psi$ is continuous. In contrast, the inequality \eqref{SGI2} is not known to be sharp, and the function $C(M,\alpha,t)$ is only partially explicit. Specifically, while the dependence on $M$, $\alpha$, and $t$ is explicitly given, the constant $c = c(\alpha, \epsilon)$ is not precisely determined.}
 \end{itemize}
}
\end{remark}

The subsequent approximation lemma will be helpful in the proof of Theorem \ref{GE}.
\begin{lemma}
\label{app}
{Let $0<q<1$ and the sequence $g_n : [0,\infty)\to [0,\infty)$ be defined as
\begin{equation}
\label{g-n}
 g_n(r)=\; \left\{
\begin{array}{cllll}
(2n)^{1-q} r \quad
&\mbox{if}&\quad 0\leq r \leq \frac{1}{2n},\\
r^{{q}}\quad&\mbox{if}&\quad
r\geq \frac{1}{2n}.
\end{array}\right.
\end{equation}
Then, the following properties are valid:
\begin{itemize}
\item[(i)]$g_n\leq g_{n+1},\quad\mbox{for all}\; n\geq1.$
\item[(ii)]$g_n(r)\leq r^q,\quad\mbox{for all}\; r\geq 0,\;n\geq 1.$
\item[(iii)]$|g_n(r)-g_n(s)|\leq (1+q)(2n)^{1-q}|r-s|,\quad\mbox{for all}\; r,s\geq 0,\;n\geq 1.$
\item[(iv)]$\displaystyle\lim_{n\to\infty}g_n(r)=r^q\quad\mbox{uniformly in}\;\;\; [0,\infty).$
\end{itemize}
}
\end{lemma}
{The proof of the aforementioned lemma can be carried out straightforwardly, and for brevity, we omit the details here.}

{The following lemma summarizes several important inequalities involving the use of the positive part notation, $[\cdot]_+$.\footnote{{Hereafter, we employ the notation $[\kappa]_+ := \max(\kappa, 0)$ for any $\kappa \in \mathbb{R}$.}}} 
\begin{lemma}
   \label{posi-part}
  {Let $g: [0,\infty)\to [0,\infty)$ be a function.
  \begin{itemize}
      \item[(i)] If $g$ is Lipschitz non-decreasing, then for every $a, b\geq 0$,
     \begin{equation}
         \label{posi-Lip}
         \left[g(a)-g(b)\right]_+\leq\,L\left[a-b\right]_+,
     \end{equation}
     where $L > 0$ denotes a Lipschitz constant for $g$.
     \item[(ii)] If $g$ is a concave function, then for every $a, b\geq 0$, 
     \begin{equation}
         \label{posi-Conc}
         \left[g(a)-g(b)\right]_+\leq\,g\left([a-b]_+\right).
     \end{equation}
  \end{itemize}
  In particular, we have 
   \begin{equation}
         \label{posi-Conc-q}
         \left[a^q-b^q\right]_+\leq\,\left([a-b]_+\right)^q,\quad 0<q<1,\quad a, b\geq 0.
     \end{equation}}
\end{lemma}
\begin{proof}[{Proof of Lemma \ref{posi-part}}]
     {To establish \eqref{posi-Lip}, we can assume, without loss of generality, that $g(a)>g(b)$. Given that $g$ is non-decreasing, it follows that $a\geq b$. As a result, we obtain that
     \begin{equation*}
        \left[g(a)-g(b)\right]_+= |g(a)-g(b)|\leq L|a-b|= L \left[a-b\right]_+.
     \end{equation*}
     We will now focus on \eqref{posi-Conc}. To begin, we will establish that $g$ is a sub-additive function, meaning that for any $a, b \geq 0$, the following inequality holds:
 \begin{equation}
     \label{sub-add}
     g(a+b)\leq g(a)+g(b).
 \end{equation} 
Exploiting the concavity of $g$ and setting $\lambda=\frac{a}{a+b}$, we obtain
\begin{align*}
g(a) &\geq \lambda\,g(a+b)+(1-\lambda)g(0), \\
g(b) &\geq (1-\lambda)\,g(a+b)+\lambda\,g(0).
\end{align*}
Consequently, due to the non-negativity of $g(0)$, we deduce that
 $$
 g(a)+g(b)\geq g(a+b)+g(0)\geq g(a+b).
 $$
 Next, we establish that $g$ is non-decreasing. Suppose, for contradiction, that $g(b) < g(a)$ for some $0 \leq a < b$. Employing the concavity of $g$, for any $x > b$, we have
 $$
 g(x)\leq \frac{g(b)-g(a)}{b-a}(x-b)+g(b).
 $$
 As $\frac{g(b) - g(a)}{b - a} < 0$, it follows that 
$$
\limsup_{x\to \infty}g(x)\leq -\infty.
$$
 This leads to a contradiction, and, consequently, we conclude that $g$ must be non-decreasing.\newline 
 Now, we return to \eqref{posi-Conc}. Without loss of generality, we may assume $g(a) > g(b)$. Since $g$ is a non-decreasing function, it follows that $a \geq b$. Applying the sub-additivity property \eqref{sub-add}, we deduce that $g(a) - g(b) \leq g(a - b)$, giving rise to  \eqref{posi-Conc} as expected.}
\end{proof}
\section{Proof of Theorem \ref{GE}}
\label{S3}
The following lemma will be useful in our proof. 
\begin{lemma}
\label{Lip}
Let $g: [0, \infty) \to [0, \infty)$ be a non-decreasing Lipschitz function with $g(0) = 0$. Assume that $0 < \gamma < \gamma_N$, where $\gamma_N$ is defined in \eqref{G.N}. 
Then, for any $0\leq u_0\in L^\infty (\R^N)$, the Cauchy problem
\begin{equation}\label{In.Lip}
\left \{ \begin{array}{rl}
u_t-\Delta u&= |x|^{-\gamma}\, g(u)\quad  \mbox{ in }\quad {\mathbb{R}^N\times (0,\infty)}, \\
u(x,0)&=u_0(x)\quad\mbox{ in }\quad \mathbb{R}^N,
\end{array}\right.
\end{equation}
has a unique non-negative solution  $u$ defined on $[0,\infty)$, i.e. $u \in C( (0,T]; L^\infty (\R^N))$ for all $T>0$ and 
$$u(t)=\mathbf{S}(t)u_0 +\int_0^t \mathbf{S_\gamma} (t-\sigma) g(u(\sigma)) d\sigma,$$
for $t\in [0, T]$. Moreover, if $u$ and $v$ are two solutions of \eqref{In.Lip} with initial values $u_0\geq v_0$, then $u\geq v$.
\end{lemma}
\begin{proof}[{Proof of Lemma \ref{Lip}}] { Consider $M>\|u_0\|_\infty$, $T>0$, and the Banach space $E_T= L^\infty((0,T); L^\infty(\mathbb{R}^N))$ equipped with the norm $\|u\|_T=\esssup_{0\leq t\leq T}\,\|u(t)\|_\infty$. We define the closed cone in $E_T$ as follows:
$$K_T=\Big\{u \in E_T;\;\; u\geq 0\quad \mbox{and}\quad \|u\|_{T} \leq M+1 \Big\}.$$
It can easily be verified that $(K_T, d)$ is a complete metric space, where $d(u,v)=\|u-v\|_T$.
Let us define the mapping  
$$
u\longmapsto \Phi[u](t):={\mathbf S}(t)u_{0}+\displaystyle\int_{0}^{t}\,{\mathbf S}_{\gamma}(t-\sigma)\,g(u(\sigma))\,d\sigma.
$$
 Our aim is to demonstrate that $\Phi(K_T) \subset K_T$ and that $\Phi$ is a contraction provided that $T$ is sufficiently small. To establish this, let $L > 0$ denote a Lipschitz constant for the function $g$. Then, using \eqref{smooth-effect} and \eqref{SS}, we can establish that for any $u,v \in E_T$, the following holds.
\begin{eqnarray*}
\|\Phi[u](t)\|_\infty&\leq&\|{\mathbf S}(t)u_{0}\|_\infty+\displaystyle\int_0^t\,\Big\|{\mathbf S}_{\gamma}(t-\sigma)\left(g(u(\sigma))-g(0)\right)\Big\|_\infty\,d\sigma\\
&\leq& \|u_0\|_\infty+C  L\displaystyle\int_0^t\, (t-\sigma)^{-\gamma/2}\,\|u(\sigma)\|_\infty\,d\sigma\\
&\leq& M+C\,L\,(M+1)\frac{T^{1-\gamma/2}}{1-\gamma/2}
\end{eqnarray*}
and 
\begin{eqnarray*}
\|\Phi[u](t)-\Phi[v](t)\|_\infty&\leq&\displaystyle\int_0^t\,\Big\|{\mathbf S}_{\gamma}(t-\sigma)(g(u(\sigma))-g(v(\sigma)))\Big\|_\infty\,d\sigma\\
&\leq& C  L\displaystyle\int_0^t\, (t-\sigma)^{-\gamma/2}\,\|u(\sigma)-v(\sigma)\|_\infty\,d\sigma\\
&\leq& C \,L\,\frac{T^{1-\gamma/2}}{1-\gamma/2}\,\|u-v\|_T.
\end{eqnarray*}
Here, $C$ represents the constant from Proposition \ref{Key}. Consequently, we get $\Phi(K_T)\subset K_T$ and $\Phi: K_T\to K_T$ is a strict contraction for sufficiently small $T>0$. Hence, $\Phi$ possesses a fixed point $0\leq u\in E_T$, which serves as a local solution.\\

Since $\mathbf{S}(\cdot) u_0 \in C((0,\infty), L^\infty(\mathbb{R}^N))$,
and by Proposition \ref{Key}
$$
\begin{array}{ll}
\left \|\int_\tau^t \mathbf{S}_\gamma(t-\sigma) g(u(\sigma)) d\sigma \right \|_\infty& \leq \| g(u)\|_{T}\int_{\tau}^t (t-\sigma)^{-\gamma/2}d\sigma\\
& \leq C \| g(u)\|_{T} \; (t-\tau)^{1-\gamma/2} \to 0
\end{array}
$$
as $(t -\tau) \to 0$, for $0\leq \tau <t\leq T$, the continuity of $u$ follows as \cite[p. 285]{Weissler2}.

Next, we establish the unconditional uniqueness. Suppose that we have two solutions $u, v \in E_T$ that both satisfy \eqref{In.Lip} in $\mathbb{R}^N\times (0,T)$, and let $w = u - v$. This allows us to write the following inequality
$$
\|w(t)\|_\infty\leq C  L\displaystyle\int_0^t\, (t-\sigma)^{-\gamma/2}\,\|w(\sigma)\|_\infty\,d\sigma.
$$
Applying Lemma \ref{SGI}, we can conclude that $w = 0$, which establishes the desired uniqueness result.

As a consequence, we can define the maximal solution $u$ on the time interval $[0, T_{\text{\text{max}}})$ of equation \eqref{In.Lip}. In order to establish $T_{\text{\text{max}}}=\infty$, we proceed by contradiction, assuming that $T_{\text{\text{max}}}<\infty$. Observing the estimate
$$
\|u(t)\|_\infty\leq\|u_0\|_\infty+C  L\displaystyle\int_0^t\, (t-\sigma)^{-\gamma/2}\,\|u(\sigma)\|_\infty\,d\sigma,
$$
and using the singular Gronwall inequality \eqref{SGI2} in combination with the fact that $\gamma<2$, we infer that 
\begin{equation*}
\|u(t)\|_{\infty}\leq\|u_0\|_\infty \,C(\gamma, L, T_{\text{max}})<\infty \quad \mbox{for a.e.}\quad 0\leq t\leq T_{\text{max}}.
\end{equation*}
Consequently, we conclude that $T_{\text{max}}=\infty$.\\
Finally, it remains to show that $u\geq v$ when $u_0\geq v_0$. This can be achieved by establishing that  $\omega(t):=[v(t)-u(t)]_+=0$. Given that $\mathbf{S}(t)(v_0-u_0)\leq 0$ and utilizing the straightforward observation
$$
\left[\mathbf{S}_\gamma(t)\varphi\right]_+\leq \mathbf{S}_\gamma(t)\left[\varphi\right]_+,
$$
alongside the inequalities \eqref{SS} and \eqref{posi-Lip}, we derive
\begin{equation}
    \label{Compar}
        \|\omega(t)\|_\infty \leq\,C\,L\displaystyle\int_0^t (t-\sigma)^{-\gamma/2}\,\|\omega(\sigma)\|_\infty\,d\sigma.
\end{equation}
By applying Lemma \ref{SGI} to \eqref{Compar}, we conclude that $\|\omega(t)\|_\infty=0$ for all $t>0$, thereby establishing $u\geq v$ as desired.
}
\end{proof}
Now we are ready to give the proof of Theorem \ref{GE}.
\begin{proof}[{Proof of Theorem \ref{GE}}]
{Let $(g_n)$ be the sequence given by Lemma \ref{app}, and consider the approximate Cauchy problems
$$
\left \{ \begin{array}{rl}
u_t-\Delta u&= |x|^{-\gamma}\, g_n(u)\quad  \mbox{ in }\quad  \mathbb{R}^N \times (0,{\infty}),\\
u({x,0})&=u_0{(x)}+1/n \quad\mbox{ in }\quad \mathbb{R}^N.
\end{array}\right.
$$
By Lemma \ref{Lip} for every $n$ there is a unique non-negative function $u_n$ defined on $[0,\infty)$ solution, that is  $u_n\in L^\infty ((0,T); L^\infty(\R^N))$ for all $T>0$ and satisfies the integral equation
$$
u_n(t)= {\mathbf S}(t)[u_{0}+1/n]+\displaystyle\int_{0}^{t}\,{\mathbf S}_{\gamma}(t-\sigma)\,g_n(u_n(\sigma))\,d\sigma.
$$
for $t\in [0,T]$. We claim that $u_n\leq u_m$ whenever $n\geq m$. Indeed, thanks to the non-negativity of $u_n$, we have $u_n(t)\geq {\mathbf S}(t)[u_{0}+1/n]\geq {\mathbf S}(t)[1/n]=1/n$. Let $n\geq m\geq 1$. From \eqref{g-n} we have that $g_n(r)=g_m(r)=r^q$ for $r\geq 1/2m$. Since $u_m\geq \frac{1}{2m}\geq \frac{1}{2n}$ and $u_n\geq \frac{1}{2n}$, we observe that $u_m$ and $u_n$ both satisfy the same integral equation with initial values $u_0+\frac{1}{m}$ and $u_0+\frac{1}{n}$, respectively. By Lemma \ref{Lip}, it follows that $u_m\geq u_n$. Therefore, by defining $u(x,t)=\displaystyle\lim_{n\to \infty} u_n(x,t)$ and the monotone convergence theorem we have that 
$$u(t)=\mathbf{S}(t)u_0+\int_0^t \mathbf{S}_\gamma(t-\sigma) u^q(\sigma) d\sigma,$$
for $t\in [0,T]$. Since $0\leq u\leq u_1$ we conclude that $u\in L^\infty((0,T); L^\infty (\mathbb{R}^N))$. 

By \eqref{S-t-C-0-1}, we know that ${\mathbf S}(\cdot) u_0 \in C([0,\infty);C_0(\mathbb{R}^N))$. To establish that $u \in C([0,\infty); C_0(\mathbb{R}^N))$, it suffices to show that $v_2 \in C([0,\infty); C_0(\mathbb{R}^N))$, where $v_2(t)$ is defined as
$$ v_2(t)=\int_0^t \mathbf{S}_\gamma(t-\sigma) u^q(\sigma) d\sigma,$$ for $t\geq 0$. In fact, by Proposition \ref{Key} (for $q_2=\infty$) we have $v_2(t) \in C_0(\mathbb{R}^N)$ for all $t>0$. Moreover, for $0\leq \tau < t< T$
$$
\begin{array}{ll}
\displaystyle \left \| \int_\tau^t\mathbf{S}_\gamma (t-\sigma) u^q(\sigma) d\sigma  \right \|_\infty &\leq C \displaystyle\|u\|^q_{L^\infty (\R^N \times (0,T))} \int_\tau ^t (t-\sigma)^{-\gamma/2} d\sigma \\
& = C \|u\|^q_{L^\infty (\R^N \times (0,T))}\, (t-\tau)^{1-\gamma/2} \to 0
\end{array}
$$
as $(t-\tau) \to 0$. Thus,  the continuity of $v_2$ follows as \cite[p. 285]{Weissler2}.  

}
\end{proof}

\section{Proof of Theorem \ref{UN}}
\label{S4}
This section is devoted to the proof of Theorem \ref{UN}. We need of the following preliminary result.
\begin{lemma} \label{Lema.subs} {Let $0<\gamma<\gamma_N$, $\gamma_N$ defined in \eqref{G.N},  $0<q<1$, and}
\begin{eqnarray*}
 w(x,t)&=&\; \left\{
\begin{array}{cllll}\lambda t^{1/(1-q)}\left (|x|+\sqrt{t}\right )^{-\gamma/(1-q)}\quad&\mbox{if}&\quad
t>0,\;x\in\R^N,\\ 0 \quad
&\mbox{if}&\quad t=0,\;x\in\R^N.
\end{array}
\right.
\end{eqnarray*}
Then, $w\in C([0,\infty); C_0(\mathbb{R}^N))$ and 
\begin{equation}\label{subs}
{w(t)\leq  \displaystyle\int_0^t {\mathbf{S}_\gamma}(t-\sigma) w^q(\sigma)d\sigma,\quad\mbox{for all}\;\; t>0.}
\end{equation}
 {Here, $\lambda^{1-q} = (1-q) \eta_0 $, with $\eta_0$ defined in accordance with \eqref{eta-0}.}
\end{lemma}
\begin{proof}[{Proof of Lemma \ref{Lema.subs}}] It is obvious that $w\in C((0,\infty); C_0(\mathbb{R}^N))$. Furthermore, we have the inequality $0\leq w(x,t) \leq \lambda t^{(1-\gamma/2)[1/(1-q)]}$. As a result, it follows that $\displaystyle\lim_{t\to 0^{+}}\|w(t)\|_{\infty}=0$, leading to the conclusion that $w\in C([0,\infty); C_0(\mathbb{R}^N))$. 
{To establish the inequality \eqref{subs}, we proceed as follows. Due to the definition of $w$, we get 
\begin{equation*}
\begin{split}
&\displaystyle\int_0^t {\mathbf{S}_\gamma}(t-\sigma) w^q(\sigma)d\sigma= \lambda^q \displaystyle\int_0^t {\mathbf S}_\gamma(t-\sigma) (|\cdot|+\sqrt{\sigma})^{-\gamma q/(1-q)} \sigma^{q/(1-q)} d\sigma\\
&=\lambda^q \displaystyle\int_0^t [4\pi (t-\sigma)]^{-N/2} \sigma^{q/(1-q)} \displaystyle\int_{\mathbb{R}^N} \exp \left( -\frac{|x-y|^2}{4(t-\sigma)}\right) |y|^{-\gamma} (|y|+\sqrt{\sigma})^{-\gamma q/(1-q)} dy d\sigma \\
&=\lambda^q (4\pi)^{-N/2} \displaystyle\int_0^t \sigma^{q/(1-q)}\int_{\mathbb{R}^N} \exp\left(-\frac{|z|^2}{4} \right) \frac{|x+\sqrt{t-\sigma} z|^{-\gamma}}{(|x+ \sqrt{t-\sigma}z|+  \sqrt{\sigma})^{\gamma q/(1-q)}} dz  d\sigma\\
&\geq \lambda^q (4\pi)^{-N/2} \displaystyle\int_0^t \sigma^{q/(1-q)} \displaystyle\int_{\mathbb{R}^N} \exp\left (-\frac{|z|^2}{4}\right)\frac{(|x|+\sqrt{t}|z|)^{-\gamma}} {(|x|+ \sqrt{t} |z|+\sqrt{t})^{\gamma q/(1-q)}}dzd\sigma.
\end{split}
\end{equation*}
}
An elementary computation gives 
$$
{|x|+\sqrt{t}|z|\leq |x|+ \sqrt{t} |z|+\sqrt{t}\leq (|x|+\sqrt{t})(1+|z|).}
$$
Thus,

\begin{eqnarray*}
{\displaystyle\int_0^t {\mathbf{S}_\gamma}(t-\sigma) w^q(\sigma)d\sigma} &\geq& \lambda^q \eta_0 \cdot  (|x|+\sqrt{t})^{-\gamma -\gamma q/(1-q)}\displaystyle\int_0^t \sigma^{q/(1-q)}d\sigma\\
&=&\lambda^q (1-q) \eta_0 \cdot (|x|+\sqrt{t})^{-\gamma/(1-q)} t^{1/(1-q)}\\
&=&w(x,t).
\end{eqnarray*}
\end{proof}
\noindent{\bf Proof of Theorem \ref{UN} (i).} The claim that $w$ is a sub-solution of problem (\ref{In.uno}) follows from Lemma \ref{Lema.subs}. To establish the estimate \eqref{Est.pr}, we adapt the arguments employed in \cite{AE87}, considering two situations: 

\noindent{\bf a)} Assume first that $u_0(x)\geq \tilde C \exp(-a|x|^2)$ for all $x\in \mathbb{R}^N$ and some constants $\tilde C, a>0$. {Subsequently, utilizing \eqref{Identity}, we obtain}
$$[{\mathbf S}(t)u_0](x)\geq \tilde C {(2\pi)^N}\,(1+4at)^{-N/2} \exp \left (\frac{-a|x|^2}{1+4a t}\right ){=\mathsf{C}}\,(1+4at)^{-N/2} \exp \left (\frac{-a|x|^2}{1+4a t}\right ),$$
where $\mathsf{C}=\tilde C (2\pi)^{N}$. Thus,
$$
\begin{array}{l}
{\mathbf S}(t-\sigma) |\cdot|^{-\gamma} u^q(\sigma)\\
\geq {\mathbf S} (t-\sigma)|\cdot|^{-\gamma}[{\mathbf S} (\sigma) u_0]^q\\
\geq\, {\mathsf{C}}^q[4\pi (t-\sigma)]^{-N/2} (1+4a\sigma)^{-N q/2} \displaystyle\int_{\mathbb{R}^N} \exp\left[ -\frac{|x-y|^2}{4(t-\sigma)}\right] |y|^{-\gamma}  \exp \left( -\frac{aq |y|^2}{1+4a\sigma}\right) dy.\\
\end{array}
$$
Employing the identity
$$\frac{|x-y|^2}{t}+\frac{|y|^2}{s}=\frac{s}{t(t+s)}\left|x-\frac{(s+t)y}{s} \right|^2+\frac{|x|^2}{t+s},$$ which is valid for all  $x,y \in \mathbb{R}^N$ and $t,s>0$, we get
\begin{equation*}
\frac{|x-y|^2}{4(t-\sigma)}+ \frac{aq|y|^2}{(1+4 a \sigma)}
=\frac{1+4a\sigma}{4(t-\sigma)h(t,\sigma)}\left|x-\frac{h(t,\sigma)}{1+4a\sigma} y\right|^2 + \frac{aq}{h(t,\sigma)}|x|^2,
\end{equation*}
where $h(t,\sigma)=4aq(t-\sigma)+1+4a\sigma.$ Thus,
\begin{equation}
\begin{array}{l}
{\mathbf S} (t-\sigma) |\cdot|^{-\gamma} u^q(\sigma)\\
\geq {\mathsf{C}}^q[4\pi (t-\sigma)]^{-N/2} (1+4a\sigma)^{-N q/2}\exp \left( - \frac{aq}{h(t,\sigma)}|x|^2\right) \cdot\\
\hskip20pt\displaystyle\int_{\mathbb{R}^N} \exp \left(-\frac{1+4a\sigma}{4(t-\sigma)h(t,\sigma)}\left|x-\frac{h(t,\sigma)}{1+4a\sigma} y\right|^2 \right)|y|^{-\gamma}dy\\
= {\mathsf{C}}^q[4\pi (t-\sigma)]^{-N/2} (1+4a\sigma)^{-N q/2}\exp \left[ - \frac{aq}{h(t,\sigma)}|x|^2\right] \cdot \left( \frac{1+4a\sigma}{h(t,\sigma)}\right)^{N-\gamma}\\
\hskip20pt\displaystyle\int_{\mathbb{R}^N} \exp \left[-\frac{1+4a\sigma}{4(t-\sigma)h(t,\sigma)}\left|x- y\right|^2 \right]|y|^{-\gamma}dy\\
\geq {\mathsf{C}}^q [4\pi (t-\sigma)]^{-N/2} (1+4a\sigma)^{-N q/2+N} h(t,\sigma)^{-N}\exp \left[ - \frac{aq}{h(t,\sigma)}|x|^2\right] \cdot\\
\displaystyle\int_{\mathbb{R}^N} \exp \left (-\frac{|x-y|^2}{4(t-\sigma)} \right )|y|^{-\gamma}dy, \\
\end{array}
\label{S.c}
\end{equation}
since $h(t,\sigma)\geq 1+4a\sigma$.
On the other hand, since $4aqt +1 \leq h(t,\sigma) \leq 4at+1$ for $0\leq \sigma \leq t$, we have
\begin{equation}
\begin{split}
u(t)& \geq \displaystyle\int_0^t {\mathbf S}(t-\sigma)|\cdot|^{-\gamma}u^q(\sigma)d\sigma\\
&\geq \mathsf{C}^q (4\pi)^{-N/2} (1+4at)^{-N}\exp \left( - \frac{aq}{4aqt+1}|x|^2\right) \int_0^t \int_{\mathbb{R}^N} \exp \left(-\frac{|w|^2}{4}\right) \cdot \\
& \hskip40pt |x+\sqrt{t-\sigma}w|^{-\gamma} dw d\sigma  \\
& \geq {{C}_1} t (1+4at)^{-N}\exp \left( - \frac{aq}{4aqt+1}|x|^2\right)\displaystyle (|x|+ \sqrt{t})^{-\gamma}\\
\end{split}
\label{S.a}
\end{equation}
where ${C_1}=\mathsf{C}^q (4\pi)^{-N/2}\int_{\mathbb{R}^N} \exp(-|w|^2/4)(1+|w|)^{-\gamma}dw\geq \mathsf{C^q}\eta_0$, and $\eta_0$ is given by (\ref{eta-0}).
We claim that
\begin{equation}
u(t)\geq C_k t^{\frac{1-q^k}{1-q}} (1+ 4aq^{k-1} t)^{-N} \exp \left (-\frac{aq^k}{4aq^k t+1}|x|^2\right )
(|x|+\sqrt{t})^{-\gamma\frac{1-q^k}{1-q}},
\label{S.b}
\end{equation}
where $C_{k}\geq {C}_1^{q^k} \eta_0^{(1-q^{k-1})/(1-q)}(1-q)^{(1-q^{k-1})/(1-q)}$ for $k\geq 1$.
From estimate \eqref{S.a}, we have \eqref{S.b} for $k=1$. Now, assume that \eqref{S.b} holds.
Arguing as the derivation of \eqref{S.c}, define $h_{k+1}(t,\sigma)=4aq^{k+1}(t-\sigma) + 4aq^k\sigma +1$, for $0<\sigma<t$. Note that 
$$4aq^{k+1}t+1 \leq h_{k+1}(t,\sigma)\leq 4aq^k t+1,$$ 
since $\partial_\sigma h_{k+1}(t,\sigma)=4aq^k(1-q)\geq 0$. {Thanks to the inequalities
$$h_{k+1}(t,\sigma)\geq 4aq^k \sigma+1,\quad 1+4aq^k \sigma \geq (1+4aq^{k-1}\sigma)^q,$$ we get}
$$
\begin{array}{l}
{\mathbf S} (t-\sigma)|\cdot|^{-\gamma} u^q(\sigma) \\
\geq C_k^q[4\pi (t-\sigma)]^{-N/2} \sigma^{q\frac{1-q^k}{1-q}} (1+4a q^{k-1}\sigma)^{-N q}\\ \displaystyle\int_{\mathbb{R}^N } \exp\left( -\frac{|x-y|^2}{4(t-\sigma)} \right)|y|^{-\gamma} \exp\left( -\frac{aq^{k+1}}{4aq^k \sigma+1}|y|^2 \right)(|y|+\sqrt{\sigma})^{-\gamma q \frac{1-q^k}{1-q}}dy\\
=C_k^q[4\pi (t-\sigma)]^{-N/2} \sigma^{q\frac{1-q^k}{1-q}} (1+4aq^{k-1}\sigma)^{-N q} \exp \left( - \frac{aq^{k+1}|x|^2}{h_{k+1} (t,\sigma)}\right)\cdot\\
\displaystyle\int_{\mathbb{R}^N} \exp \left(- \frac{4aq^k\sigma+1}{4(t-\sigma) h_{k+1}(t,\sigma)} \left |x- \frac{h_{k+1}(t,\sigma)}{4aq^k\sigma+1}y\right |^2\right) |y|^{-\gamma} (|y|+\sqrt{\sigma})^{-\gamma q \frac{1-q^{k}}{1-q}}dy\\
\geq C_k^q[4\pi (t-\sigma)]^{-N/2} \sigma^{q\frac{1-q^k}{1-q}}  (1+4aq^{k-1}\sigma)^{-N q} \exp \left( - \frac{aq^{k+1}|x|^2}{h_{k+1} (t,\sigma)}\right)\cdot\\
\displaystyle\int_{\mathbb{R}^N} \exp \left(- \frac{1}{4(t-\sigma)} \left |x- \frac{h_{k+1}(t,\sigma)}{4aq^k\sigma+1}y\right |^2\right) |y|^{-\gamma} (|y|+\sqrt{\sigma})^{-\gamma q \frac{1-q^{k}}{1-q}}dy\\
\geq C_k^q[4\pi (t-\sigma)]^{-N/2} \sigma^{q\frac{1-q^k}{1-q}} (1+4aq^{k-1}\sigma)^{-N q} \left (\frac{4aq^k\sigma+1}{h_{k+1}(t,\sigma)}\right )^{N-\gamma} \exp \left( - \frac{aq^{k+1}|x|^2}{h_{k+1} (t,\sigma)}\right) \\
\displaystyle\int_{\mathbb{R}^N} \exp \left( -\frac{|x-z|^2}{4(t-\sigma)}\right)  |z|^{-\gamma}\left (\frac{4aq^k\sigma+1}{h_{k+1}(t,\sigma)}|z|+\sqrt{\sigma} \right)^{-\gamma q \frac{1-q^{k}}{1-q}}dz\\
\geq C_k^q [4\pi(t-\sigma)]^{-N/2}  \sigma^{q\frac{1-q^k}{1-q}} (1+ 4aq^{k-1} \sigma)^{-N q}(1+4aq^{k}\sigma)^{N} (1+4aq^kt)^{-N} \exp \left (-\frac{aq^{k+1}|x|^2}{4aq^{k+1} t+1}\right) \\
\displaystyle\int_{\mathbb{R}^N} \exp \left( -\frac{|x-z|^2}{4(t-\sigma)}\right)  |z|^{-\gamma}\left (|z|+\sqrt{\sigma} \right)^{-\gamma q \frac{1-q^{k}}{1-q}}dz\\
\geq C_k^q (4\pi)^{-N/2}  \sigma^{q\frac{1-q^k}{1-q}}  (1+4aq^kt)^{-N} \exp \left (-\frac{aq^{k+1}|x|^2}{4aq^{k+1} t+1}\right )  \\
\displaystyle\int_{\mathbb{R}^N} \exp(-\frac{|w|^2}{4})|x+\sqrt{t-\sigma} w|^{-\gamma} (|x+ \sqrt{t-\sigma} w|+\sqrt{\sigma})^{-\gamma q \frac{1-q^k}{1-q}}dw\\
\geq C_k^q (4\pi)^{-N/2}  \sigma^{q\frac{1-q^k}{1-q}}  (1+4aq^kt)^{-N} \exp \left (-\frac{aq^{k+1}|x|^2}{4aq^{k+1} t+1}\right )  \\
\displaystyle\int_{\mathbb{R}^N} \exp(-\frac{|w|^2}{4})(|x|+\sqrt{t} |w|)^{-\gamma} (|x|+ \sqrt{t} |w|+\sqrt{t})^{-\gamma q \frac{1-q^k}{1-q}}dw\\
\geq C_k^q \eta_k  \sigma^{q\frac{1-q^k}{1-q}}  (1+4aq^kt)^{-N} \exp \left (-\frac{aq^{k+1}|x|^2}{4aq^{k+1} t+1}\right )  (|x|+\sqrt{t})^{-\gamma \frac{1-q^{k+1}}{1-q}},\\
\end{array}
$$
{where
$$\eta_k=(4\pi)^{-N/2} \displaystyle\int_{\mathbb{R}^N} \exp \left(-\frac{|w|^2}{4}\right)  (1+|w|)^{-\gamma  \frac{1-q^{k+1}}{1-q}} \geq \eta_0.$$
} Therefore,
$$u(t)\geq C_{k+1} t^{\frac{1-q^{k+1}}{1-q}} (1+4aq^k t)^{-N} \exp(-\frac{aq^{k+1}|x|^2}{4aq^{k+1}t+1}) (|x|+\sqrt{t})^{-\gamma \frac{1-q^{k+1}}{1-q}},$$
where
$$C_{k+1}=\eta_k C_k^q  \frac{1-q}{1-q^{k+1}}\geq (1-q) \eta_0  C_k^q.$$
Hence, we obtain
$C_{k+1}\geq {{C_1}}^{q^{k+1}} \eta_0^{(1-q^k)/(1-q)} (1-q)^{(1-q^k)/(1-q)}.$ Therefore, estimate (\ref{S.b}) holds for $k+1$ and the induction process is complete. Letting $k\to \infty$ we obtain the desired estimate (\ref{Est.pr}).\\

\noindent{\bf b)} {In the general case, consider $\epsilon > 0$, and define $u_\epsilon(t)$ as $u(t+\epsilon)$. Based on the integral equation \eqref{integral} and thanks to \eqref{S-t-0}, we infer $u_\epsilon(0)=u(\epsilon) \geq \mathbf{S}(\epsilon)u_0\geq C \exp(-a |x|^2)$ for some positive constants $C$ and $a$. By applying item (i), we  conclude that $u_\epsilon(t)\geq [(1-q) \eta_0 ]^{1/(1-q)} t^{1/(1-q)} (|x|+\sqrt{t})^{-\gamma/(1-q)}.$ Hence, we obtain the following inequality:
$$u(t)=u((t-\epsilon)+\epsilon)\geq [(1-q) \eta_0]^{1/(1-q)}\, (t-\epsilon)^{1/(1-q)}\, (|x|+\sqrt{t-\epsilon})^{-\gamma/(1-q)}.$$
As a result, by taking the limit as $\epsilon$ approaches $0$, the desired estimate \eqref{Est.pr} is obtained.}\\
\noindent{\bf Proof of Theorem \ref{UN} (ii).} {The uniqueness can be deduced from the lemma presented below.}
\begin{lemma} \label{Lem.PC}{ Consider $0 < q < 1$ and non-negative functions $u, v \in C([0,\infty);C_0(\mathbb{R}^N))$ satisfying the inequalities:
$$
\left \{ \begin{array}{rl}
u(t) \geq {\mathbf S}(t)u_0+\displaystyle\int_0^t \mathbf{S}_\gamma(t-\sigma) u^q(\sigma)d\sigma,\\\\
v(t) \leq {\mathbf S}(t)v_0+\displaystyle\int_0^t \mathbf{S}_\gamma(t-\sigma) v^q(\sigma)d\sigma,
\end{array}\right.
$$
where $u_0 \geq v_0 \geq 0$ and $u_0 \neq 0$. Then, there exists $\gamma^* > 0$ such that $u(t) \geq v(t)$ for all $t \geq 0$ whenever $0 < \gamma < \gamma^*.$}
\end{lemma}
\begin{proof}[{Proof of Lemma \ref{Lem.PC}}]
Define $z(t) = [v(t) - u(t)]_+$ for $t \geq 0$. Employing \eqref{posi-Conc-q}, we obtain
$$z(t)  \leq \displaystyle\int_0^t {\mathbf{S}}_\gamma(t-\sigma) [v^q(\sigma)-u^q(\sigma)]_{+}d\sigma \leq \displaystyle\int_0^t {\bf S}_\gamma(t-\sigma)z^q(\sigma) d\sigma.$$

Thus, by Lemma \ref{maxim}
$$
\|z(t)\|_{\infty}\leq \eta_1 \bigg(\displaystyle\int_0^t (t-\sigma)^{-\gamma/2} d\sigma\bigg) \left( \sup_{0\leq \sigma \leq t} \{\|z(\sigma)\|_{\infty}\}\right)^q,
$$
where 
\begin{equation} \label{gamma-1}
\eta_1=\eta_1(\gamma)=(4\pi)^{-N/2}\displaystyle\int_{\mathbb{R}^N} \exp \left( -\frac{|y|^2}{4}\right)|y|^{-\gamma}dy.  
\end{equation}

Hence,
\begin{equation}\label{Cec.un}
\|z(t)\|_{\infty} \leq \left ( \frac{2\eta_1}{2-\gamma} t^{1-\gamma/2}\right)^{1/(1-q)}.
\end{equation}

On the other hand, by the mean value theorem there exists $u_\theta=\theta u+(1-\theta)v$ with $\theta \in (0,1)$ such that
$v^q(\sigma)-u^q(\sigma)=q u_\theta^{q-1}(\sigma)[v(\sigma)-u(\sigma)].$ Thus, by the estimate (\ref{Est.pr})

$$
\begin{array}{ll}
z(t)& \leq \displaystyle\int_0^t {\mathbf{S}}_\gamma(t-\sigma) [v^q(\sigma)-u^q(\sigma)]_{+}d\sigma\\
&\leq q\displaystyle\int_0^t \mathbf{S}_\gamma (t-\sigma) \left\{ u(\sigma)^{q-1} [v(\sigma)-u(\sigma))]_{+}\right\}\\
&\leq  q [ (1-q) \eta_0]^{-1} \displaystyle\int_0^t \sigma^{-1} \mathbf{S}_\gamma(t-\sigma)  (|\cdot|+ \sqrt{\sigma})^{\gamma}z(\sigma) d\sigma.
\end{array}
$$
Hence,
\begin{equation}\label{Mar.z}
    \|z(t)\|_\infty \leq  q[(1-q) \eta_0 ]^{-1}\displaystyle\int_0^t \sigma^{-1} \|{\mathbf{S}}_\gamma (t-\sigma)(|\cdot|+\sqrt{\sigma})^\gamma\|_\infty \|z(\sigma)\|_\infty d\sigma.
\end{equation}

Note that
$$
\begin{array}{rl}
{\mathbf S}_\gamma(t-\sigma)(|\cdot|+\sqrt{\sigma})^{{\gamma}}&=[4\pi (t-\sigma)]^{-N/2}\displaystyle\int_{\mathbb{R}^N} \exp \left[-\frac{|x-y|^2}{4(t-\sigma)} \right] \left(1+ \frac{\sqrt{\sigma}}{|y|}\right)^{\gamma}dy,\\
\end{array}
$$
and therefore, from Lemma \ref{Pinsky},
$$
\begin{array}{ll}
\|{\mathbf{S}}_\gamma (t-\sigma)(|\cdot|+\sqrt{\sigma})^\gamma\|_\infty&= \left[ {\mathbf{S}}_\gamma (t-\sigma)(|\cdot|+\sqrt{\sigma})^\gamma \right] (0)\\
&=[4\pi (t-\sigma)]^{-N/2} \displaystyle\int_{\mathbb{R}^N} \exp \left[ -\frac{|y|^2}{4(t-\sigma)}\right] \left( 1+\frac{\sqrt{\sigma}}{|y|}\right)^{\gamma} dy\\
&=(4\pi)^{-N/2} \displaystyle\int_{\mathbb{R}^N} \exp\left(-|y|^2/4\right)\left( 1+ \frac{\sqrt{\sigma}}{\sqrt{t-\sigma}|y|} \right)^{\gamma}{dy}.
\end{array}
$$
To estimate this last integral we argue as follows. For $|y|\geq 1$ and $0<\sigma<t$, we have $$\left( 1+ \frac{\sqrt{\sigma}}{\sqrt{t-\sigma}|y|}\right)^{\gamma}\leq (t-\sigma)^{-\gamma/2} (\sqrt{t-\sigma}+ \sqrt{\sigma})^{\gamma}\leq (2t)^{\gamma/2}(t-\sigma)^{-\gamma/2},$$
while for  $|y|\leq 1$,
$$\left( 1+ \frac{\sqrt{\sigma}}{\sqrt{t-\sigma}|y|}\right)^{\gamma}\leq |y|^{-\gamma}(t-\sigma)^{-\gamma/2} (\sqrt{t-\sigma}+ \sqrt{\sigma})^{\gamma}\leq |y|^{-\gamma} (2t)^{\gamma/2}(t-\sigma)^{-\gamma/2}.$$
Therefore,
$$
\|{\mathbf S}_\gamma (t-\sigma)(|\cdot|+\sqrt{\sigma})^{\gamma}\|_{\infty} \leq \eta_2 \;  t^{\gamma/2}(t-\sigma)^{-\gamma/2},
$$
where 
\begin{equation}\label{eta-2}
\begin{array}{lll}
\eta_2&=&\eta_2(\gamma)\\
&=& (4\pi)^{-N/2}2^{\gamma/2} \left [ \displaystyle\int_{|y|\geq 1}  \exp \left(-\frac{|y|^2}{4} \right)dy+ \displaystyle\int_{|y|\leq 1} \exp \left( -\frac{|y|^2}{4}\right) |y|^{-\gamma}dy \right].
\end{array}
\end{equation}
Inserting this in \eqref{Mar.z} we obtain
$$
\|z(t)\|_{\infty} \leq q [(1-q) \eta_0 ]^{-1} \eta_2 \; t^{\gamma/2}\displaystyle\int_0^t \sigma^{-1}(t-\sigma)^{-\gamma/2}\|z(\sigma)\|_{\infty}d\sigma.
$$
 Using estimate \eqref{Cec.un}, we have
$$
\begin{array}{rl}
\|z(t)\|_{\infty}&\leq  q [(1-q) \eta_0]^{-1}\eta_2 \; t^{\gamma/2} \left (\frac{2\eta_1}{2-\gamma} \right )^{1/(1-q)} \displaystyle\int_0^t \sigma^{-1+ (2-\gamma)/2(1-q)} (t-\sigma)^{-\gamma/2}d\sigma\\
&=q [(1-q) \eta_0 ]^{-1} \eta_2 \left (\frac{2\eta_1}{2-\gamma} \right)^{1/(1-q)} t^{(2-\gamma)/2(1-q)} \displaystyle\int_0^1 \sigma^{-1+ (2-\gamma)/2(1-q)} (1-\sigma)^{-\gamma/2}d\sigma\\
&=q [ (1-q) \eta_0]^{-1} \eta_2 \beta(\gamma) \left (\frac{2\eta_1}{2-\gamma} t^{(2-\gamma)/2}\right)^{1/(1-q)},
\end{array}
$$
where $\eta_1$ is given in \eqref{gamma-1} and 
\begin{equation}
\label{beta-gamma}
\beta(\gamma)=\displaystyle\int_0^1 \sigma^{-1+ (2-\gamma)/2(1-q)} (1-\sigma)^{-\gamma/2}d\sigma.
\end{equation}
 Arguing similarly $k$ times, we obtain
\begin{equation}
\label{uniq-conc}
\| z(t)\|_\infty \leq \left(\Lambda(\gamma)\right)^k\,\left (\frac{2\eta_1}{2-\gamma} t^{(2-\gamma)/2}\right)^{1/(1-q)},
\end{equation}
where
$$
\Lambda(\gamma)=\frac{q\beta(\gamma)\eta_2(\gamma)}{(1-q)\eta_0(\gamma)},
$$
and $\eta_0(\gamma),\, \eta_2(\gamma),\, \beta(\gamma)$  are respectively given by \eqref{eta-0}, \eqref{eta-2} and \eqref{beta-gamma}.
Note that one can easily verify using the Lebesgue theorem that $\Lambda(\gamma)\to q<1$ as $\gamma\to 0$. Letting $k\to \infty$ in \eqref{uniq-conc} yields the desired result for $\gamma>0$ being sufficiently small.
This finishes the proof of Lemma \ref{Lem.PC}, thereby concluding the proof for Theorem \ref{UN}.
\end{proof}
\begin{remark}
{\rm It is not clear how to remove the assumption of smallness on $\gamma$ to obtain uniqueness. However, we conjecture that uniqueness holds across the entire range of $0 \leq \gamma < \gamma_N=\min\{2, N\}$.}
\end{remark}

\section{{Conclusion and Open Problems}}

{In this work, we established the existence and uniqueness of nonnegative global solutions for the nonlinear heat equation with an inhomogeneous concave nonlinearity in $\mathbb{R}^{N}$. Our results extend previous ones by considering the case where the nonlinearity is influenced by a singular potential $|x|^{-\gamma}$, with $0 < \gamma < \gamma_N$. We proved the existence of solutions for initial data in $C_0(\mathbb{R}^{N})$ and provided a lower bound for non-trivial solutions, which was instrumental in establishing the uniqueness for small values of $\gamma$.}

{Several open problems remain. Although we established uniqueness for small $\gamma$, it remains unclear whether uniqueness holds for all $0 < \gamma < \gamma_N$. A deeper understanding of the behavior of solutions as $\gamma$ approaches $\gamma_N$ is needed. Furthermore, the regularity of solutions, especially near the singularity at $x = 0$, is an interesting direction for future research. Understanding how the singularity affects the smoothness of solutions could provide further insight into the structure of the problem. Investigating the long-time behavior of solutions, including convergence to steady states or blow-up phenomena, would also be valuable. This could involve studying the influence of the singular potential on the decay rates of solutions. Finally, extending the results to more general nonlinearities or to other types of singular potential could broaden the applicability of the methods developed in this work.}

\vspace{0.1cm}

\hrule

\vspace{0.1cm}

\noindent{\bf\large Acknowledgements.} {{\em 
The authors would like to express their sincere thanks to the Editor for the careful and efficient handling of the manuscript. They are also deeply grateful to the anonymous Referees for their meticulous review and valuable suggestions that have greatly improved the quality of this work.
Miguel Loayza was partially supported by CAPES-PRINT, 88881.311964/2018-01, MATHAMSUD, 88881.520205/2020-01, CNPq - 313382/2023-9 .}}\\

 \hrule

\vspace{0.1cm}

{\noindent{\bf\large Declarations.}}
On behalf of all authors, the corresponding author states that there is no conflict of interest. No data-sets were generated or analyzed during the current study.

\vspace{0.1cm}

\hrule 

\end{document}